\documentclass[12pt]{amsart}
\usepackage{amsmath, amstext, amsbsy, amssymb}

\newtheorem{theorem}{Theorem}[section]

\newtheorem{proposition}{Proposition}[section]

\theoremstyle{definition}

\theoremstyle{remark}

\numberwithin{equation}{section}

\renewcommand{\a}{\alpha}

\newcommand{\ep}{\epsilon}
\newcommand{\vep}{\varepsilon}

\newcommand{\Z}{\mathbb Z}

\def\Za{\mathbb Z+1/2}
\def\Zp{\mathbb Z_++1/2}

\def\a{\alpha}

\begin{document}

\title[Fermionic realization of toroidal Lie algebras]
{Fermionic realization of toroidal Lie algebras of classical types}
\author{Naihuan Jing and Kailash C. Misra}
\address{Department of Mathematics,
   North Carolina State Univer\-sity,
   Ra\-leigh, NC 27695-8205, USA}
\email{jing@math.ncsu.edu, misra@math.ncsu.edu}
\thanks{Jing acknowledges the support of NSA grant H98230-06-1-0083 and NSFC grant 10728102,
and Misra acknowledges the support of NSA grant H98230-06-1-0025.}
\keywords{Clifford algebras, vertex operators}
\subjclass{Primary: 17B, 20}

\begin{abstract}
We use fermionic operators to construct toroidal Lie algebras of
classical types, including in particular that of symplectic affine
algebras, which is first realized by fermions.
\end{abstract}

\maketitle

\section{Introduction} \label{S:intro}

Toroidal Lie algebras are natural generalization of the affine
Kac-Moody algebras \cite{MRY} that enjoy many similar interesting
features. Let $\mathfrak{g}$ be a finite-dimensional complex
simple Lie algebra of type $X_n$, and $R=\mathbb{C}[s,s^{-1},t,t^{-1}]$ be the
ring of Laurent polynomials in commuting
variables $s$ and $t$. By definition a 2-toroidal Lie algebra of type $X_n$ is a perfect central extension of the iterated loop algebra
$\mathfrak{g}\otimes R$.

Let $\Omega_R/dR$ be the K\"ahler differentials of $R$ modulo the
exact forms. The universal central extension of the iterated loop
algebra is given by $T(X_n)=(g\otimes R)\oplus
\Omega_R/dR$. Any 2-toroidal Lie algebra of type $X_n$ is a homomorphic image of
this toroidal Lie algebra. The center of $T(X_n)$ is $\Omega_R/dR$,
which is a infinite-dimensional vector space. The Laurent
polynomial ring $R$ induces a natural $\mathbb{Z}^2-$gradation on
$T(X_n)$. The center is given by
$\Omega_R/dR=\oplus_{\sigma\in
\mathbb{Z}^2}{\mathcal{Z}}(g)_{\sigma}$, with
dim${\mathcal{Z}}_{\sigma}=1$ if $\sigma\not= (0,0)$, and 2 if
$\sigma=(0,0)$. We denote by $c_0$ and $c_1$ the two standard
degree zero central elements in the toroidal Lie algebra
$T(X_n)$. A module of $T(X_n)$ is called a
level-$(k_0,k_1)$ module if the standard pair of central elements $(c_0,c_1)$ acts as $(k_0,k_1)$ for some complex numbers $k_0$ and $k_1$. In this
work we will only consider modules with $k_0\neq 0$.

Representation theory of toroidal Lie algebras have been studied
extensively in recent years (for example see: \cite{MRY} \cite{BB},
\cite{EM}, \cite{FM}, \cite{T1}, \cite{T2}, \cite{JMT}, \cite{XH}). Some
exceptional types were studied in \cite{LH}.
Most of these
realizations are bosonic. In \cite{FF} bosonic and fermionic
constructions for the classical affine Lie algebras are given.
Motivated by \cite{FF}, Gao \cite{G} gave both bosonic and fermionic
constructions of the extended affine general linear algebras. In
\cite{L}, Lau gave more general bosonic and fermionic constructions
which included as special cases the constructions in \cite{FF} and
\cite{G} for the affine and extended affine general linear algebras
as well as Virasoro algebra and $W$-algebra.

In this paper we extend the fermionic construction of
Feingold-Frenkel \cite{FF} to give a unified fermionic construction
for all 2-toroidal Lie algebras of classical types. In particular,
it includes the special case of $q=1$ in \cite{G}. Our main idea is
to construct the operators corresponding to special nodes in the
affine Dynkin diagrams. It turns out that the idea of ghost fields
plays an important role in our construction. By introducing special
auxiliary fields we are able to define actions for the root vectors
corresponding to the special nodes.

Another new feature of our construction is that we have succeeded in
realizing the type $C$ toroidal Lie algebras exclusively using
fermions, which contains the affine symplectic Lie algebras as a
special case. In the original Feingold-Frenkel construction the type
$C$ case was not available in fermionic construction, and it has
been an open problem for nearly 25 years. This is because one needs
to use quadratic expressions to represent all root vectors in the
algebra. If $b(z)$ is a fermionic field, then $:b(z)b(z):=0$, thus
it is impossible to directly use fermionic operators to realize the
special node $\alpha_n=2\vep_n$ in type $C$ toroidal Lie algebras,
and we have found new ways to realize the symplectic affine Lie
algebras by introducing new fermions. In a subsequent paper we will
generalize this idea to the bosonic picture \cite{JMX}.

The structure of this paper is as follows. In section 2 we define
the toroidal Lie algebra, and state MRY-presentation \cite{MRY} of
the toroidal algebra in terms of generators and relations. In
section 3 we start with a finite rank lattice with a symmetric
bilinear form and define a Fock space and some vertex operators,
which in turn give level (1,0) representations of the toroidal Lie
algebra of type $A_n, B_n$, $C_n$ or  $D_n$. The proof is an
extensive analysis of the operator product expansions for the field
operators. We also include the verification of the Serre relations.

\section{Toroidal Lie Algebras}
Let $\mathfrak{g}$ be the complex simple Lie algebra over
$\mathbb{C}$ of type $X_n$, and $R=\mathbb{C}[s,s^{-1},t,t^{-1}]$ be the ring of
Laurent polynomials in commuting variables. We consider the
iterated loop algebra $\mathfrak{g}(R)=\mathfrak{g}\otimes R$. A
toroidal Lie algebra of type $X_n$ is a perfect central extension
of the iterated loop algebra $\mathfrak{g}(R)$. Let $\Omega_R$ be the
$R$-module of differentials with differential mapping $d:$ $R\to
\Omega_R$, such that $d(g_1g_2)=(dg_1)g_2+g_1(dg_2)$ for all
$g_1,g_2$ in $R$. Let $-:$ $\Omega_R\to\Omega_R/dR$ be the
canonical linear map for which ${\overline{dg}}=0$ for all $g\in
R$. Define the vector space
$$
T(X_n):=(\mathfrak{g}\otimes R)\oplus \Omega_R/dR
$$
with the following bracket operation defined by
$$
[x\otimes g_1, y\otimes g_2]=[x,y]\otimes g_1g_2+(x,y){\overline{g_2dg_1}},
$$
and $\Omega_R/dR$ central, for $x,y\in \mathfrak{g}$, $g_1,g_2\in
R$, where $(\cdot,\cdot)$ is the trace form. From \cite{MRY} we
know that $T(X_n)$ is a perfect Lie algebra and is the universal
central extension of the iterated loop algebra
$\mathfrak{g}(R)$. Therefore, any toroidal Lie algebra of
type $X_n$ is a homomorphic image of $T(X_n)$. The gradation of
the polynomial ring $R$ gives a natural $\mathbb{Z}^2$-gradation
to the toroidal Lie algebra $T(X_n):=\oplus_{\sigma\in
\mathbb{Z}^2}T(X_n)_{\sigma}$, where $T(X_n)_{\sigma}$ is spanned
by $x\otimes s^{m_0}t^{m_1}$, ${\overline
{s^{m_0}t^{m_1}s^{-1}ds}}$ and ${\overline
{s^{m_0}t^{m_1}t^{-1}dt}}$ for $\sigma=(m_0,m_1)\in \mathbb{Z}^2$
and $x\in \mathfrak{g}$. The condition ${\overline {dg}}=0$ for
all $g\in R$ implies that $m_0{\overline
{s^{m_0}t^{m_1}s^{-1}ds}}+ m_1{\overline
{s^{m_0}t^{m_1}t^{-1}dt}}=0$ for all $m_0,m_1\in \mathbb{Z}$.
Therefore dim$T(X_n)_{\sigma}=1+ \mbox{dim}(\mathfrak{g})$ if $\sigma\not=(0,0)$, and $2+ \mbox{dim}(\mathfrak{g})$ if
$\sigma=(0,0)$. In particular, $T(X_n)_{(0,0)}$ is spanned by
$x\otimes 1$ for $x\in \mathfrak{g}$, and central elements
${\overline{s^{-1}ds}}$, ${\overline {t^{-1}dt}}$. We denote these
two degree zero central elements by $c_0$ and $c_1$.

The most interesting quotient algebra of the toroidal Lie algebra
$T(X_n)$ is the double affine algebra, denoted by $T_0(X_n)$, that
is the toroidal Lie algebra of type $X_n$ with a two dimensional
center. The double affine algebra is the quotient of $T(X_n)$
modulo all the central elements with degree other than zero. In
fact, $T_0(X_n)$ has the following realization
$$
T_0(X_n)=(\mathfrak{g}\otimes R)\oplus \mathbb{C}c_0\oplus
\mathbb{C}c_1
$$
with the Lie product
$$
[x\otimes g_1, y\otimes g_2]=[x,y]\otimes g_1g_2+\Phi(g_2\partial_sg_1)c_0
+\Phi(g_2\partial_tg_1)c_1
$$
for all $x,y\in \mathfrak{g}$, $g_1,g_2\in R$, where $\Phi$ is a
linear functional on $R$  defined by $\Phi(s^kt^m)=0$
if $(k,m)\not= (0,0)$
and $\Phi(s^kt^m)=1$, if $(k,m)=(0,0)$
for all $k,m\in \mathbb{Z}$.

\vspace{3mm}

\noindent {\bf Definition 2.1} {\it If $M$ is a module for a
toroidal Lie algebra of type $X_n$, we call $M$ a
level-$(k_0,k_1)$ module for some complex numbers $k_0,k_1$ if the
degree zero central elements $c_0$, $c_1$ act on $M$ as constants
$k_0,k_1$ respectively.} \vspace{3mm}

In the present paper we want to give concrete construction of
level-$(1,0)$ module for the toroidal Lie algebra $T(X_n)$, and also
for the double affine algebra $T_0(X_n)$, for $X= A,B,C,D$.

Let $(a_{ij})_{{n+1}\times {n+1}}$ be the generalized Cartan matrix of the
affine algebra $X_n^{(1)}$, and $Q:=
\mathbb{Z}\alpha_0 \oplus \mathbb{Z}\alpha_1  \oplus \cdots \oplus \mathbb{Z}\alpha_n$ its root lattice. The
toroidal Lie algebra $T(X_n)$  has the following presentation
\cite{MRY} with generators $\not c$, $\alpha_i(k)$,  and
$x_k(\pm\alpha_i)$ for $i=0,1, \cdots , n$, $k\in \mathbb{Z}$, and the
following relations:
 \vspace{3mm}

   (R0). $[{\not c}, \alpha_i(k)]=0=[{\not c}, x_k(\pm\alpha_i)];$

   (R1). $[\alpha_i(k),\alpha_j(m)]=k(\a_i | \a_j)\delta_{k+m,0}{\not c};$

   (R2). $[\alpha_i(k), x_m(\pm\alpha_j)]=\pm (\a_i | \a_j) x_{k+m}(\pm\alpha_j);$

   (R3).   $[x_k(\alpha_i),x_m(-\alpha_j)]=-\delta_{ij}\frac{2}{(\a_i | \a_j)}\left\{\alpha_i(k+m)+k\delta_{
   k+m,0}{\not c}\right\};$

   (R4). $[x_k(\alpha_i), x_m(\alpha_i)]=0=[x_k(-\alpha_i), x_m(-\alpha_i)];$

   $\;\;\;\;\;\;\left(ad x_0(\alpha_i)\right)^{-a_{ij}+1}x_m(\alpha_j)=0,$ if $i\not= j$;

   $\;\;\;\;\;\;\left(ad x_0(-\alpha_i)\right)^{-a_{ij}+1}x_m(
   -\alpha_j)=0,$ if $i\not= j$;
 \\
   \\
for $i,j=0,1, \cdots , n$ and $k,m\in \mathbb{Z}$. It is known \cite{MRY} that there is an isomorphism $\psi$ between the two presentations of $T(X_n)$. In this paper we will identify the two presentations of the toroidal Lie algebra $T(X_n)$ via this isomorphism $\psi$. In particular, under this isomorphism we can identify the degree zero central elements $c_0={\not c}$ and $c_1= \delta(0)$, where $\delta$ is the null root in $Q$. We also remark that in our
definition of the toroidal Lie algebra, we use the root operators $\alpha_i(k)$
instead of coroot operators $\alpha_i^{\vee}(k)$ as in \cite{MRY}.

  Following \cite{MRY}, we introduce a $\mathbb{Z}\times Q$-gradation on $T(X_n)$ by assigning deg${\not c}=(0,0)$, deg$\alpha_i(k)=(k,0)$,
 deg$x_k(\pm\alpha_i)=(k,\pm\alpha_i)$, with $i=0,1, \cdots , n$ and $k\in \mathbb{Z}$.
We denote by $T_k^{\alpha}$ the subspace of $T(X_n)$ spanned by
the elements with degree $(k,\alpha)$ for $k\in \mathbb{Z}$,
$\alpha\in Q$. Then, under the isomorphism $\psi$, we have
$\psi^{-1}({\overline{s^kt^{-1}dt}})=\delta(k)\in T^0_k$, and
$\psi^{-1}({ \overline{s^kt^{r}s^{-1}ds}})\in T^{r\delta}_k$.

 Let $z,w,z_1,z_2,...$ be formal variables. We define formal power series with
 coefficients from the toroidal Lie algebra $T(X_n)$:
 $$
 \alpha_i(z)=\sum_{n\in \mathbb{Z}}\alpha_i(n)z^{-n-1},
\qquad x(\pm \alpha_i,z)=\sum_{n\in \mathbb{Z}}x_n(\pm\alpha_i)z^{-n-1},
 $$
 for $i=0,1, \cdots , n$.
We will use the delta function
\begin{equation*}
\delta(z-w)=\sum_{n\in\Z}w^nz^{-n-1}
\end{equation*}

Using $\displaystyle \frac1{z-w}=\sum_{n=0}^{\infty}z^{-n-1}w^n$, $|z|>|w|$,
we have the following useful expansions:
\begin{align*}
\delta(z-w)&=\iota_{z, w}((z-w)^{-1})+\iota_{w, z}((w-z)^{-1}),\\
\partial_w\delta(z-w)&=
\iota_{z, w}((z-w)^{-2})-\iota_{w, z}((w-z)^{-2}),
\end{align*}
where $\iota_{z, w}$ means expansion when $|z|>|w|$. For simplicity
in the following we will drop $\iota_{z, w}$ if it is
clear from the context.

Now the Lie algebra structure of $T(X_n)$ can be expressed in terms of the following power series identities:

 \vspace{3mm}
   $(R0)^\prime$ $[{\not c}, \alpha_i(z)]=0=[{\not c}, x(\pm\alpha_i,z)];$

   $(R1)^\prime$ $[\alpha_i(z),\alpha_j(w)]=
(\a_i | \a_j)\partial_w\delta(z-w){\not c};$

   $(R2)^\prime$ $[\alpha_i(z), x(\pm\alpha_j, w)]=
\pm (\a_i | \a_j)x(\pm\alpha_j,w)\delta (z-w);$

   $(R3)^\prime$ $[x(\alpha_i,z),x(-\alpha_j,w)]=
-\delta_{ij}\frac{2}{(\a_i | \a_j)}\left\{ \alpha_i(w)\delta(z-w)+
\partial_w\delta(z-w){\not c}\right\};$

   $(R4)^\prime$ $[x(\alpha_i,z),x(\alpha_i,w)]=0=
[x(-\alpha_i,z),x(-\alpha_i,w)]$,\\
and for $0 \le i\not= j \le n$,

   $\mathrm{ad}x(\pm\alpha_i,z_1)x(\pm\alpha_j,z_2)=0,$
   if $a_{ij}=0$

   $(\mathrm{ad}x(\pm\alpha_i,z_1))(\mathrm{ad}x(\pm\alpha_i,z_2))
   x(\pm\alpha_j,z_3)=0,$
   if $a_{ij}=-1$

   $(\mathrm{ad}x(\pm\alpha_i,z_1))(\mathrm{ad}x(\pm\alpha_i,z_2))
(\mathrm{ad}x(\pm\alpha_i,z_3))x(\pm\alpha_j,z_4)=0,$

if $a_{ij}=-2$.

\section{Representations of the Toroidal Algebra}
In this section we give a fermionic realizations for the toroidal
Lie algebra of classical types $A_{n-1}, B_n$, $D_n$ and $C_n$.

Let $\vep_i$ $(i=0, \ldots, n+1$) be a set of orthonormal basis of the vector space
$\mathbb C^{n+2}$ equipped with the innner product $( \ | \ )$ such that
\begin{equation*}
(\vep_i|\vep_j)=\delta_{ij},
\end{equation*}
Let $P_0=\mathbb Z\vep_1\oplus \cdots \oplus \mathbb Z\vep_{n}$
be a sublattice of rank $n$, and let
$\overline{c}=\frac1{\sqrt{2}}(\vep_0+i\vep_{n+1})$ correspond to
the null vector $\delta$ and $\overline{d}=
\frac1{\sqrt{2}}(\vep_0-i\vep_{n+1})$ be the dual gradation
operator. Then
\begin{align*}
(\overline{c}|\overline{c})&=(\overline{c}|\vep_i)=0\\
(\overline{d}|\overline{d})&=(\overline{d}|\vep_i)=0\\
(\overline{c}|\overline{d})&=1
\end{align*}
for $i=1, \ldots, n$.

The simple roots for the classical finite dimensional
Lie algebras can be realized
simply by defining the simple roots as follows:

\medskip
$\a_1=\vep_1-\vep_2$, $\cdots$, $\a_{n-1}=\vep_{n-1}-\vep_n$; for $A_{n-1}$.

$\a_1=\vep_1-\vep_2$, $\cdots$, $\a_{n-1}=\vep_{n-1}-\vep_n$,
$\a_n=\vep_n$; for $B_n$.

$\a_1=\frac{\vep_1-\vep_2}{\sqrt 2}$, $\cdots$,
$\a_{n-1}=\frac{\vep_{n-1}-\vep_n}{\sqrt 2}$, $\a_n={\sqrt
2}\vep_n$; for $C_n$.

$\a_1=\vep_1-\vep_2$, $\cdots$, $\a_{n-1}=\vep_{n-1}-\vep_n$,
$\a_n=\vep_{n-1}+\vep_n$; for $D_n$.

\medskip

Then the set of positive roots are:

$$
\Delta_+=\begin{cases} \{\vep_i-\vep_j|1\le i<j \le n\}, & \text{Type $A_{n-1}$}\\
\{\vep_i, \vep_i\pm\vep_j|1 \le i<j \le n\}, & \text{Type $B_n$}\\
\{\sqrt{2}\vep_i, \frac1{\sqrt 2}(\vep_i\pm\vep_j)|1 \le i<j \le
n\}, & \text{Type $C_n$}\\
 \{\vep_i\pm\vep_j|1 \le i<j \le n\}, &
\text{Type $D_n$}.
\end{cases}
$$

The highest (long) root $\a_{max}$ for each type is given as follows:

$$
\a_{max}=\begin{cases} \vep_1-\vep_n, & \text{Type $A_{n-1}$}
\\
\sqrt 2\vep_1 & \text{Type $C_n$}\\ \vep_1+\vep_2, & \text{Type
$B_n$ or $D_n$}

\end{cases}
$$

We further introduce the element
$$\a_0=\overline{c}-\a_{max}$$
in the lattice and then define $\beta=-\overline{c}+\vep_1$ for type
$ABD$, and $\beta=-\sqrt2\overline{c}+\vep_1$ for type $C$. Then we
have
$$
\a_0=\vep_{n}-\beta\ \ \text{for $A_{n-1}$}; \quad -\beta-\vep_2 \ \
\text{for $B_n, D_n$}; \quad \text{or} -\frac1{\sqrt
2}(\beta+{\vep_1}) \ \ \text{for $C_n$}.
$$
Note that $(\beta|\beta)=1, (\beta|\vep_i)=\delta_{1i}$.

Then $P=\mathbb Z\overline{c}\oplus\mathbb Z\vep_1\oplus \cdots
\oplus \mathbb Z \vep_n$ and $Q=\mathbb Z\overline{c}\oplus\mathbb
Z\a_1\oplus \cdots \oplus \mathbb Z \a_k= \mathbb Z\a_0\oplus \cdots
\oplus \mathbb Z \a_k$ ($k= n-1$ for type $A_{n-1}$ and $k= n$ for
types $B_n,D_n$) are the weight lattice and root lattice for the
corresponding affine Lie algebra, and $P_0=\mathbb Z\vep_1\oplus
\cdots \oplus \mathbb Z \vep_k$ and $Q_0=\mathbb Z\a_1\oplus \cdots
\oplus  \mathbb Z\a_k$ ($k= n-1$ for type $A_{n-1}$ and $k= n$ for
types $B_n,D_n$) are the weight lattice and root lattice for the
simple Lie algebras $A_{n-1}, B_n, C_n, D_n$. Then
$(\a_i|\a_j)=d_ia_{ij}$, where $a_{ij}$ are the entries of the
affine Cartan matrix of type $(ABCD)^{(1)}$, and the $d_i$'s are
given by:
$$
(d_0, d_1, \cdots, d_k) = \begin{cases} \{1, 1, \cdots, 1,1\}, & , k=n-1,\text{Type $A_{n-1}$}\\
\{1, 1, \cdots, 1,\frac{1}{2}\}, & , k=n,\text{Type $B_n$}\\
\{1, \frac{1}{2}, \cdots, \frac{1}{2},1\}, & k=n,\text{Type $C_n$}\\
\{1, 1, \cdots, 1,1\}, & , k=n,\text{Type $D_n$}.
\end{cases}
$$

\medskip

We introduce infinite dimensional Clifford algebras as follows. We
first let $\vep_{\overline i}$, $1\leq i\leq n$ be orthonormal
vectors such that $(\vep_{\overline i}|\vep_{\overline
j})=\delta_{ij}$ and $(\vep_{\overline i}|\vep_{j})=0$. Let
$\tilde{P}_{\mathbb C}$ be the $\mathbb C$-vector space spanned by
$\overline c$ and $\vep_i, 1 \le i \le n$ for types $AD$, by
$\overline c$ and $\vep_i, \vep_{\overline i}, 1 \le i \le n$ in
type $C$, and by $\overline c$, $\vep_i, 1 \le i \le n$ and a ghost
element $e$ (to be defined later) for types $B$. We also denote
$\overline\beta=-\sqrt2 \overline c+\vep_{\overline 1}$. Then we
define $\mathcal C=\mathcal C_0\oplus \mathcal C_1$, where both
subspaces $\mathcal C_0=\tilde{P}_{\mathbb C}$ and $\mathcal
C_1=\tilde{P}_{\mathbb C}^*$ are maximal isotropic subspaces. The
symmetric bilinear form on $\mathcal C$ is given by
$$
<b^*, a>=<a, b^*>=(a|b), \quad <a, b>=<a^*, b^*>=0, \qquad a, b\in
\tilde{P}_{\mathbb C}
$$

In this way we have a maximal polarization of $\mathcal C$.

The Clifford algebra $Cl(\tilde{P})$
is generated by the central element $C$
and the elements $a(k)$ and $a^*(k)$,
where $a \in \tilde{P}_{\mathbb C}= \tilde{P}\otimes {\mathbb C}$,
$a^*\in \tilde{P}_{\mathbb C}^*$, and $k\in \Za$ subject to the relations:
\begin{align*}
\{a(k), b(l)\}&=0 \\
\{a^*(k), b^*(l)\}&=0 \\
\{a(k), b^*(l)\}&=(a|b)\delta_{k, -l}C
\end{align*}
where $a, b\in \tilde{P}_{\mathbb C}$.
Note that the
anticommutation relations can be simply written as
$$
\{u(k), v(l)\}=<u, v>\delta_{k, -l}C, \qquad u, v\in \mathcal C.
$$

The representation space is the infinite dimensional vector space
$$\displaystyle
V=\bigotimes_{a_i}\left(\bigotimes_{k\in\Zp} \mathbb C[a_i(-k)]
\bigotimes_{k\in\Zp} \mathbb C[a_i^*(-k)]\right)
$$
where $a_i$ runs through any basis in $\tilde P$,
say $\overline{c}$, $\vep_i$'s and $\vep_{\overline i}$'s.

The Clifford algebra acts on the space $V$ by the usual action:
$a(-k)$ acts as a creation operator, $a(k)$ as an annihilation operator
and $C$ acts as the identity.

For any two fermionic fields $u(z)=\sum_{n\in\Za}a(n)z^{-n-1/2}$
and $v(z)=\sum_{n\in\Za}b(n)z^{-n-1/2}$ we define
the normal ordering $:u(z)v(w):$ by their components:
\begin{equation}
:u(m)v(n):=\begin{cases} u(m)v(n) & m<0\\
-v(n)u(m) & m>0\end{cases}.
\end{equation}

It follows from the definition that the normal ordering satisfies
the relation:
$$:u(z)v(w): \,= -:v(w)u(z):.$$

Based on the normal product of two fields, we can
define the normal product of $n$ fields inductively as follows. We
define that

\begin{align*}
&:u_1(z_1)u_2(z_2)\cdots u_n(z_n):\\
&=:u_1(z_1)(:u_2(z_2)\cdots u_n(z_n):):,
\end{align*}
and then use induction till we reach 2 fields.

We define the contraction of two states by
$$
\underbrace{a(z)b(w)}=a(z)b(w)-:a(z)b(w):,
$$
which contains all poles for $a(z)b(w)$. In general, the contraction
of several pairs of states is given inductively by
the following rule.

\begin{proposition} \label{OPE} The basic operator product expansions are:
for $x, y\in{\mathcal C}$ we have
$$
\underbrace{u(z)v(w)}=\frac{<u, v>}{z-w}
$$
In particular we have for $a, b\in P_{\mathbb C}$
\begin{align*}
\underbrace{a(z)b(w)}&=\underbrace{a^*(z)b^*(w)}=0,\\
\underbrace{a(z)b^*(w)}&=\underbrace{a^*(z)b(w)}
=\frac{(a, b)}{z-w}, \qquad
\end{align*}
\end{proposition}
\begin{proof}
In fact one has
\begin{align*}
a(z)b^*(w)&=:a(z)b^*(w):+\sum_{m\in \Zp}[a(m), b^*(n)]z^{-m-1/2}w^{-n-1/2}\\
&=:a(z)b^*(w):+(a, b)C\sum_{0<m\in\Zp}z^{-m-1/2}w^{-m-1/2}\\
&=:a(z)b^*(w):+\frac{(a, b)}{z-w}.
\end{align*}
The other OPEs are proved in the same manner.
\end{proof}

\begin{proposition} \label{anticommutation} The fermionic fields satisfy the following
anticommutation relations:
\begin{align*}
\{a(z), b(w)\}&=\{a^*(z), b^*(w)\}=0,\\
\{a(z), b^*(w)\}&=(a, b)\delta(z-w).
\end{align*}
\end{proposition}
\begin{proof} As in the proof of Proposition
\ref{OPE}, we have
\begin{align*}
\{u(z), v(w)\}&=u(z)v(w)+v(w)u(z)\\
&=<u, v>\left(\frac1{z-w}+\frac1{w-z}\right) \\
&=<u, v>\delta(z-w)
\end{align*}
since $:u(z)v(w):+:v(w)u(z):= 0$.
\end{proof}

The following result is easily obtained by Wick's theorem.
\begin{proposition} \label{bracket} The brackets among normal order products are given by:
\begin{align*}
&[:r_1(z)r_2(z):,:s_1(w)s_2(w):]=\\
&\quad<r_1, s_2>:r_2(z)s_1(z):\delta(z-w)-<r_1, s_1>:r_2(z)s_2(z):\delta(z-w)\\
&\, \, +<r_2, s_1>:r_1(z)s_2(z):\delta(z-w)-<r_2, s_2>:r_1(z)s_1(z):\delta(z-w)\\
&\quad+(<r_1, s_2><r_2, s_1>-<r_1, s_1><r_2, s_2>)\partial_w\delta(z-w).\\
\end{align*}
\end{proposition}

The inner product of the underlying Lie algebra
can be extended to that of the linear factors as follows:

$$<:r_1r_2:, :s_1s_2:>= - <r_1, s_1><r_2, s_2>
 +<r_1, s_2><r_2, s_1>$$

For any root vector $\a=\sum_{i\in I}\vep_i-\sum_{j\in I'}\vep_j$,
we define the field operator $X(\a, z)$ as follows.
\begin{equation}
X(\a, z)= :\prod_{i\in I}^{\rightarrow}\vep_i(z)\prod_{j\in
I'}^{\rightarrow}\vep_j^*(z):,
\end{equation}
where $\vec{\prod}_{i\in I}$ means the ordered product
according to the natural order in $I$.

Furthermore, we introduce a ghost field $e(z)=\sum_{k\in\mathbb Z+1/2}e(k)z^{-k-1/2}$
such that
\begin{align*}
(e|e)=-1 &, (e|\vep_i)=0\\
\{e(k), e(l)\}&=-\delta_{k, -l}
\end{align*}

Then we have
\begin{align*}
\underbrace{e(z)\vep_i(w)}&=\underbrace{e(z)\vep_i^*(w)}=0\\
\underbrace{e(z)e(w)}&=\frac{-1}{z-w}
\end{align*}

\begin{theorem} Under the following map we have a
level (1,0) representation of the toroidal Lie algebra of classical
types:
\begin{align*}
X(\a_0, z)&=\begin{cases}:\vep_n(z)\beta^*(z): &
\text{for $A_{n-1}$}\\
\frac1{\sqrt 2}(:\beta^*(z)\vep_{\overline 1}^*(z):-:\vep_{1}^*(z)\overline{\beta}^*(z):) & \text{for $C_n$}\\
:\beta^*(z)\vep_2^*(z): & \text{for $B_n$ or $D_n$}\end{cases}\\
X(-\a_0, z)&=\begin{cases}:\vep_n^*(z)\beta(z): &
\text{for $A_{n-1}$}\\
\frac1{\sqrt 2}(:\beta(z)\vep_{\overline 1}(z):-:\vep_{
1}(z)\overline{\beta}(z):)
& \text{for $C_n$}\\
:\beta(z)\vep_2(z): & \text{for $B_n$ or $D_n$}\end{cases}\\
X(\a_i, z)&=\begin{cases}:\vep_i(z)\vep_{i+1}^*(z): \quad
1\leq i\leq n-1, &\text{for $A_{n-1}B_nD_n$}\\
:\vep_i(z)\vep_{i+1}^*(z):-:\vep_{\overline{i+1}}(z)\vep_{\overline
i}^*(z):
& \text{for $C_n$}\end{cases}\\
X(-\a_i, z)&=\begin{cases}:\vep_i^*(z)\vep_{i+1}(z):, \quad
1\leq i\leq n-1, &\text{for $A_{n-1}B_nD_n$}\\
:\vep_i^*(z)\vep_{i+1}(z):-:\vep_{\overline{i+1}}^*(z)\vep_{\overline
i}(z):
& \text{for $C_n$}\end{cases}\\
X(\a_n, z)&=\begin{cases}
\sqrt{2}:\vep_n(z)e(z): & \text{for $B_n$}\\
:\vep_n(z)\vep^*_{\overline n}(z): & \text{for $C_n$}\\
:\vep_{n-1}(z)\vep_n(z): & \text{for $D_n$}\\
\end{cases}\\
X(-\a_n, z)&=\begin{cases}
\sqrt{2}:e(z)\vep_n^*(z): & \text{for $B_n$}\\
:\vep^*_{n}(z)\vep_{\overline n}(z): & \text{for $C_n$}\\
:\vep_{n-1}^*(z)\vep_n^*(z): & \text{for $D_n$}\\\end{cases}
\end{align*}
and the fields for the simple roots are represented by
\begin{align*}
\a_0(z)&=\begin{cases}:\vep_n(z)\vep_n^*(z):-
:\beta(z)\beta^*(z): & \text{for $A_{n-1}$}\\
:\beta^*(z)\beta(z):+:\vep_2^*(z)\vep_2(z):& \text{for $B_n$ or
$D_n$}\\
\frac12(:\beta^*(z)\beta(z):+:\vep_1^*(z)\vep_1(z): & \\
\quad +:\overline{\beta}^*(z)\overline{\beta}(z):
+:\vep_{\overline{1}}^*(z)\vep_{\overline{1}}(z))&  \text{for
$C_{n}$}
\end{cases}\\
\a_i(z)&=\begin{cases}
:\vep_i(z)\vep_i^*(z):-:\vep_{i+1}(z)\vep_{i+1}^*(z): &
1\leq i\leq n-2,\\
\frac12(:\vep_i(z)\vep_i^*(z):-:\vep_{i+1}(z)\vep_{i+1}^*(z): & \\
\quad -:\vep_{\overline i}(z)\vep_{\overline i}^*(z):
+:\vep_{\overline{i+1}}(z)\vep_{\overline{i+1}}^*(z):)& \text{for
$C_n$}
\end{cases}
\end{align*}
\begin{align*}
\a_n(z)&=\begin{cases}
:\vep_n(z)\vep_n^*(z):& \text{for $B_n$}\\
:\vep_n(z)\vep_n^*(z):-:\vep_{\overline n}(z)\vep_{\overline n}^*(z):& \text{for $C_n$}\\
:\vep_{n-1}(z)\vep_{n-1}^*(z):+:\vep_n(z)\vep_n^*(z):& \text{for
$D_n$}
\end{cases}
\end{align*}
\end{theorem}

\begin{proof}

First for type $X\neq C$ and $i, j=1, \ldots, n-1 \; (n-2 \;
\text{for} \; A_{n-1})$, using Proposition \ref{bracket} we have
\begin{align*}
&[X(\a_i, z), X(-\a_j, w)]\\
&=-[:\vep_i(z)\vep_{i+1}^*(z):,:\vep_{j+1}(w)\vep_j^*(w):]\\
&=-\delta_{ij}\left((:\vep_{i+1}^*(z)\vep_{i+1}(z):+:\vep_i(z)\vep_{i}^*(z):)
\delta(z-w)+\partial_w\delta(z-w)\right) \\
&=-\delta_{ij}\left(\a_i(z)\delta(z-w)+\partial_w\delta(z-w)\right)
\end{align*}

Similarly we have for $i, j=1, \ldots, n-1 \; (n-2 \; \text{for} \; A_{n-1})$
\begin{align*}
&[\a_i(z),\a_j(w)]\\
&=[:\vep_i(z)\vep_i^*(z):-:\vep_{i+1}(z)\vep_{i+1}^*(z):,:\vep_j(z)\vep_j^*(z):-:\vep_{j+1}(z)\vep_{j+1}^*(z):]\\
&=2\delta_{ij}\partial_w\delta(z-w)-\delta_{i+1,j}\partial_w\delta(z-w)-\delta_{i,j+1}\partial_w\delta(z-w)\\
&=a_{ij}\partial_w\delta(z-w)=(\a_i | \a_j)\partial_w\delta(z-w),
\end{align*}
where $a_{ij}$ are the entries of type $A$ Cartan matrix.

More generally using Proposition \ref{bracket} we have
\begin{align*}
&[X(\vep_i-\vep_j, z), X(\vep_k-\vep_l, w)]\\
&=\delta_{jk}X(\vep_i-\vep_l, z)\delta(z-w)
-\delta_{li}X(\vep_k-\vep_j, z)\delta(z-w)\\
&\qquad\qquad +\partial_w\delta(z-w).
\end{align*}

Now for any $k\neq l$ we have
\begin{align*}
&[\alpha_i(z), X(\vep_k-\vep_l, w)]\\
&=[:\vep_i(z)\vep_i^*(z):-:\vep_{i+1}(z)\vep_{i+1}^*(z):,
:\vep_k(w)\vep_l^*(w)]\\
&=(\delta_{ik}-\delta_{il}-\delta_{i+1,k}+\delta_{i+1,l})
X(\vep_k-\vep_l, z)\delta(z-w)\\
&=(\a_i|\vep_k-\vep_l)
X(\vep_k-\vep_l, z)\delta(z-w)
\end{align*}

It then follows that
the operators $X(\a_i, z)$, $X(-\a_i, z)$ generate a type $A$
subalgebra for $i=1, \ldots, n-1\; (n-2 \; \text{for} \; A_{n-1})$.

In type $C$, similar computation gives for $1\leq i, j\leq n-1$
$$[X(\a_i, z), X(-\a_j, w)]
=-2\delta_{ij}\left(\a_i(z)\delta(z-w)+\partial_w\delta(z-w)\right),
$$
and
\begin{align*}
&[\a_i(z),\a_j(w)]\\
&=d_i^{-1}d_j^{-1}2(2\delta_{ij}-\delta_{i+1,j}-\delta_{i,j+1})\partial_w\delta(z-w)\\
&=(\alpha_i|\alpha_j)\partial_w\delta(z-w),
\end{align*}

We now check for the special nodes. First for type $A_{n-1}$, if we
start with $i=0$ or $i=n-1$ the above relation is still valid
provided we take the indices modulo $n$ and treating $\beta(z)$ as
$\vep_1(z)$. In particular for type $A_{n-1}$ we have
\begin{align*}
&[X(\a_{n-1}, z), X(-\a_0, w)]\\
&=-[:\vep_{n-1}(z)\vep_0^*(z):, :\vep_n^*(w)\beta(w):]\\
&= 0.
\end{align*}
\begin{align*}
&[X(\a_0, z), X(-\a_0, w)]=[:\vep_n(z)\beta^*(z):,:\vep_n^*(w)\beta(w):]\\
&=-\left(:\vep_n^*(z)\vep_n(z):\delta(z-w)-:\beta(z)\beta^*(z):\right)
\delta(z-w)-\partial_w\delta(z-w) \\
&=-\left(\a_0(z)\delta(z-w)+\partial_w\delta(z-w)\right)
\end{align*}

Furthermore, in this case we have
\begin{align*}
&[\a_{n-1}(z),\a_0(w)]\\
&=[:\vep_{n-1}(z)\vep_{n-1}^*(z):-:\vep_n(z)\vep_n^*(z):,
:\vep_n(w)\vep_n^*(w):
-:\beta(w)\beta^*(w):]\\
&=-[:\vep_n(z)\vep_n^*(z):,:\vep_n(w)\vep_n^*(w):]\\
&=-\partial_w{\delta(z-w)}=(\a_{n-1}| \a_0)\partial_w{\delta(z-w)}
\end{align*}

In type $B_n$ case we first have that for $k\neq l$
\begin{align*}
&[\a_0(z), X(\ep_k-\ep_l, w)]=-[:\beta(z)\beta^*(z):+
:\vep_2(z)\vep_2^*(z),:\vep_k(w)\vep_l^*(w):]\\
&=(-\ep_1-\ep_2|\ep_k-\ep_l)X(\ep_k-\ep_l, w)\delta(z-w)\\
&=(\a_0|\ep_k-\ep_l)X(\ep_k-\ep_l, w)\delta(z-w)
\end{align*}

\begin{align*}
&[X(\a_0, z), X(-\a_1, w)]=[:\beta^*(z)\vep_2^*(z):,:\vep_1^*(w)\vep_2(w):]\\
&=-:\beta^*(z)\vep_1^*(w):\delta(z-w)=-:\vep_1^*(z)\vep_1^*(z):\delta(z-w)=0,
\end{align*}
where we have used the property that $\beta(z)=-\overline{c}(z)+\vep_1(z)$ and
\newline  $:\overline{c}(z)u(z):=0$ for any $u\in \mathcal C$.

\begin{align*}
&[X(\a_0, z), X(-\a_0, w)]=[:\beta^*(z)\vep_2^*(z):,:\beta(w)\vep_2(w):]\\
&=-(:\beta^*(z)\beta(w):+:\vep_2^*(z)\vep_2(w):)\delta(z-w)-\partial_w\delta(z-w)\\
&=-(\a_0(z)\delta(z-w)+\partial_w\delta(z-w))
\end{align*}

\begin{align*}
&[X(\a_{n-1}, z), X(-\a_n, w)]\\
&=[:\vep_{n-1}(z)\vep_{n}^*(z):,:e(w)\vep_n^*(w):]=0
\end{align*}

In this case we also have
\begin{align*}
&[X(\a_{n}, z), X(-\a_n, w)]\\
&=2[:\vep_{n}(z)e(z):,:e(w)\vep_{n}^*(w):]\\
&=-2\left(:\vep_n(z)\vep_n^*(w):\delta(z-w)-:e(z)e(w):\delta(z-w)
+\partial_w\delta(z-w)\right) \\
&=-2\left(\a_n(z)\delta(z-w) +\partial_w\delta(z-w)\right),
\end{align*}
where we have used the fact that $:e(z)e(z):=0$.

Now for type $D_n$ case the relations involving $i=0$ are exactly
same as in type $B_n$ case. For the other relations we have:
\begin{align*}
&[X(\a_{n-1}, z), X(-\a_n, w)]\\
&=[:\vep_{n-1}(z)\vep_{n}^*(z):,:\vep_{n-1}^*(w)\vep_n^*(w):]\\
&= -:\vep_n^*(z)\vep_n^*(w):\delta(z-w)\\
&= -:\vep_n^*(z)\vep_n^*(z):\delta(z-w)= 0
\end{align*}

\begin{align*}
&[X(\a_{n}, z), X(-\a_n, w)]\\
&=[:\vep_{n-1}(z)\vep_n(z):,:\vep_{n-1}^*(w)\vep_n^*(w):]\\
&=-\left((:\vep_n(z)\vep_n^*(w):+:\vep_{n-1}(z)\vep_{n-1}^*(w):)\delta(z-w)+\partial_w\delta(z-w)\right) \\
&=-\left(\a_n(z)\delta(z-w) +\partial_w\delta(z-w)\right),
\end{align*}

For type $C_n$ we have
\begin{align*}
&[X(\a_0, z), X(-\a_0, w)]=\frac12[:\beta^*(z)\vep_{\overline 1}^*(z):,:\beta(w)\vep_{\overline 1}(w):]\\
&\qquad\qquad +\frac12[:\vep_1^*(z)\overline{\beta}^*(z):,:\vep_1^*(w)\overline{\beta}^*(w):]\\
&=-\left(\a_0(z)\delta(z-w)+\partial_w\delta(z-w)\right).
\end{align*}
It is easy to see that $[X(\alpha_0, z), X(-\alpha_1,
w)]=[X(\alpha_1, z), X(-\alpha_0, w)]=0$. Moreover one has

\begin{align*}
&[X(\a_n, z), X(-\a_n, w)]=[:\vep_n(z)\vep_{\overline n}^*(z):,:\vep_n^*(w)\vep_{\overline n}(w)::]\\
&-(:\vep_{\overline n}^*(z)\vep_{\overline n}(z):+ :\vep_n(z)\vep_n^*(z))\delta(z-w)-\partial_w\delta(z-w)\\
&=-\left(\a_n(z)\delta(z-w)+\partial_w\delta(z-w)\right).
\end{align*}

As for the Serre relations, we first notice that it is easy to
check that the OPE expansions of $X(\a_i, z)X(\a_i,
w)$ or $X(-\a_i, z)X(-\a_i, w)$ are analytic for all $i=0, \ldots,
n$, thus $[X(\pm\a_i, z), X(\pm\a_i, w)]=0$.

For $i=j\pm 1$, we have:
\begin{align*}
&[X(\a_i, z_1), [X(\a_i, z_2), X(\a_j, w)]]\\
&=[X(\a_i, z_1),
\left(X(\a_i+\a_j,w)\delta_{i+1,j}-X(\a_j+\a_i,w)\delta_{j+1,
i}\right)\delta(z_2-w)\\
&\qquad\qquad +\delta_{i, j+1}\delta_{i+1,
j}\partial_w\delta(z_2-w)]\\
&=(X(2\a_i+\a_j, w)\delta_{i+1,
i}\delta_{i,j+1}-X(\a_j+2\a_i, w)\delta_{j+1,i}\delta_{j,
i+1})\delta(z_1-w)\\
&\qquad\qquad \cdot\delta(z_2-w)\\
&=0.
\end{align*}

Thus we we have shown most cases and we include
the verification for type $B$ or $D$ to show the
method for the other vertices.

\begin{align*}
&[X(\a_0, z_1), [X(\a_0, z_2), X(\a_2, w)]]\\
&=[X(\a_0, z_1),[:\beta^*(z_2)\vep_2^*(z_2):, :\vep_2(w)\vep_3^*(w):]\\
&=[:\beta^*(z_1)\vep_2^*(z_1):,:\beta^*(z_2)\vep_3^*(w):]\delta(z_2-w)=0.
\end{align*}

\begin{align*}
&[X(\a_{n-1}, z_1), [X(\a_{n-1}, z_2), X(\a_n, w)]]\\
&=[X(\a_{n-1}, z_1),[:\vep_{n-1}(z_2)\vep_n^*(z_2):, :\vep_n(w)e(w):]\\
&=[:\vep_{n-1}(z_1)\vep_n^*(z_1):,:\vep_{n-1}(z_2)e(z_2):]\delta(z_2-w)\\
&=0.
\end{align*}

\begin{align*}
&[X(\a_{n}, z_1), [X(\a_{n}, z_2),[X(\a_n, z_3), X(\a_{n-1}, w)]]]\\
&=[X(\a_{n}, z_1), [X(\a_{n}, z_2),[:\vep_{n}(z_3)e(z_3):, :\vep_{n-1}(w)\vep_n^*(w):]]]\\
&=[X(\a_{n}, z_1), [:\vep_n(z_2)e(z_2):,:e(z_3)\vep_{n-1}(z_3):]]\delta(z_3-w)\\
&=-[X(\a_{n}, z_1),
:\vep_n(z_2)\vep_{n-1}(z_2):]\delta(z_2-z_3)\delta(z_3-w)\\
&=-[:\vep_n(z_1)e(z_1):,
:\vep_n(z_2)\vep_{n-1}(z_2):]\delta(z_2-z_3)\delta(z_3-w)\\
&=0
\end{align*}
The remaining relations follow similarly.
\end{proof}

\bibliographystyle{amsalpha}

\end{document}